\documentclass[12pt,reqno]{amsart}

\usepackage{amsmath,amssymb,amsthm,amsfonts,hyperref}

\DeclareMathSymbol{\twoheadrightarrow} {\mathrel}{AMSa}{"10}

\def\Q{{\mathbb Q}}

                                              \def\mm{{\mathfrak m}}

\def\Z{{\mathbb Z}}

\def\F{{\mathbb F}}

                          \def\tors{\mathrm{tors}}

\def\Hom{\mathrm{Hom}}

\def\M{{\mathcal M}}

\def\fchar{\mathrm{char}}

\def\rk{\mathrm{rk}}

        \def\K_a{\bar{K}}

\def\M{\mathrm{M}}

\def\dim{\mathrm{dim}}

\def\K{{\mathcal{K}}}

\newtheorem{thm}{Theorem}[section]

\newtheorem{lem}[thm]{Lemma}

\newtheorem{cor}[thm]{Corollary}

\theoremstyle{definition}

\newtheorem{defn}[thm]{Definition}

\newtheorem{rem}[thm]{Remark}

\newtheorem{sect}[thm]{}

\title[Poincar\'e duality]{Poincar\'e duality and unimodularity}

\author{Yuri G. Zarhin}

\address{Department of Mathematics, Pennsylvania
State University, University Park, PA 16802, USA}

\email{zarhin\char`\@math.psu.edu}

\begin{document}
\begin{abstract}
It is well known that the cup-product pairing on the complementary
integral cohomology groups (modulo torsion) of a compact oriented
manifold is unimodular. We prove a similar result for the
$\ell$-adic cohomology groups of smooth algebraic varieties.
\end{abstract}

\maketitle

This is a corrected version of my paper \cite{ZarhinVdG} that was published in 2012; it turns out that \cite[Lemma 2.5]{ZarhinVdG} is wrong. Nevertheless the main results of \cite{ZarhinVdG} (Theorem \ref{main} and Corollary \ref{proper}) remain valid. Here we present their proofs that 
are slight modifications of arguments from \cite{ZarhinVdG}.
I am grateful to Thomas Geisser, who pointed out that Lemma 2.5 of  \cite{ZarhinVdG} is wrong.

\section{Introduction}
\label{intro}
 If  $\Lambda$  is a commutative ring with $1$  without zero
divisors and $M$ is a $\Lambda$-module,  then we write $M_{\tors}$
for its torsion submodule and $M/\tors$ for the quotient
$M/M_{\tors}$. Usually, we will use this notation when $\Lambda$ is
the ring of integers $\Z$ or the ring $\Z_{\ell}$ of $\ell$-adic
integers.

Let $\M$ be a compact connected oriented $d$-dimensional manifold
with boundary $\partial \M$. Let $i \le d$ be a nonnegative integer
and $H^i(\M;\Z)$ and $H^i(\M, \partial \M;\Z)$ are the corresponding
integral cohomology groups, which are finitely generated
$\Z$-modules. The cup-product pairing
$$H^i(\M;\Z) \times H^{d-i}(\M, \partial \M;\Z) \to \Z$$
gives rise to the pairing of free $\Z$-modules
$$H^i(\M;\Z)/\tors \times H^{d-i}(\M, \partial \M;\Z)/\tors \to \Z.$$
It is well known \cite[Ch. VI, Sect. 9, Th. 9.4 on pp.
357--358]{Bredon} that the latter pairing is {\sl perfect} or {\sl
unimodular} or {\sl duality pairing},  i.e., the induced
homomorphism
$$H^i(\M;\Z)/\tors \to \Hom_{\Z}( H^{d-i}(\M, \partial \M;\Z)/\tors,
\ \Z)$$ is an isomorphism.

The aim of this note is to prove a $\Z_{\ell}$-variant of this
result for the \'etale cohomology groups of a smooth algebraic
variety over an algebraically closed field of arbitrary
characteristic.

\begin{sect}
\label{prelimX}
Let $X$ be a separated connected smooth scheme of finite type over a
algebraically closed field $k$, $d=\dim(X)$, $\ell$ a prime $\ne
\fchar(k)$ and $a$ an integer. If  $i \le 2d$ is a nonnegative
integer then we write $H^i(X,\Z_{\ell}(a))$ and
$H_c^i(X,\Z_{\ell}(a))$ for the corresponding (twisted) $i$th
\'etale $\ell$-adic cohomology group and cohomology group with
compact support. It is known \cite[pp.
22--24]{Katz} that both $H^i(X,\Z_{\ell}(a))$ and
$H_c^i(X,\Z_{\ell}(a))$ are finitely generated $\Z_{\ell}$-modules;
in particular, both $H^i(X,\Z_{\ell}(a))/\tors$ and
$H_c^i(X,\Z_{\ell}(a))/\tors$ are free $\Z_{\ell}$-modules of finite
rank. If $b$ is an integer with $a+b=d$ then the cup-product pairing
$${H}^i(X,\Z_{\ell}(a))\times H_c^{2d-i}(X,\Z_{\ell}(b)) \to
H_c^{2d}(X,\Z_{\ell}(d))=\Z_{\ell}$$ gives rise to the
$\Z_{\ell}$-bilinear pairing of free $\Z_{\ell}$-modules of finite
rank
$${H}^i(X,\Z_{\ell}(a))/\tors\times H_c^{2d-i}(X,\Z_{\ell}(b))/\tors \to
\Z_{\ell},$$ which is known to be nondegenerate, i.e., its right and
left kernels are zero. In particular, the induced homomorphism of
free $\Z_{\ell}$-modules
$$H^i(X,\Z_{\ell}(a))/\tors \to
\Hom_{\Z_{\ell}}(H_c^{2d-i}(X,\Z_{\ell}(b))/\tors,\ \Z_{\ell})$$
is injective and its cokernel is finite \cite[Ch. 2, Sect. 1, p.
149]{FK}.
This implies that the ranks of free $\Z_{\ell}$-modules $H^i(X,\Z_{\ell}(a))/\tors$ and
$H_c^{2d-i}(X,\Z_{\ell}(b))/\tors,\ \Z_{\ell})$ do {\sl coincide}.
\end{sect}

\begin{thm}
\label{main}
Let $X$ be a separated connected smooth scheme of finite type over a
algebraically closed field $k$, $d=\dim(X)$, $\ell$ a prime $\ne
\fchar(k)$ , $a$ and $b$  integers
with $a+b=d$. Let $i \le 2d$ be a nonnegative
integer.

The pairing
$${H}^i(X,\Z_{\ell}(a))/\tors\times H_c^{2d-i}(X,\Z_{\ell}(b))/\tors \to
\Z_{\ell}$$ is perfect, i.e., the induced homomorphism
$$H^{i}(X,\Z_{\ell}(a))/\tors \to
\Hom_{\Z_{\ell}}(H_c^{2d-i}(X,\Z_{\ell}(b))/\tors, \ \Z_{\ell})$$
is an isomorphism of $\Z_{\ell}$-modules.
\end{thm}

When $X$ is complete (e.g., projective),
$H_c^{2d-i}(X,\Z_{\ell}(b))=H^{2d-i}(X,\Z_{\ell}(b))$   and we
obtain the following corollary.

\begin{cor}
\label{proper}
 Let $X$ be a separated connected $d$-dimensional complete smooth
scheme of finite type over an algebraically closed field $k$. Let
$\ell$ be a prime different from $\fchar(k)$, $a$ and $b$  integers
with $a+b=d$. If $i \le 2d$ is a nonnegative integer then the
pairing
$$H^i(X,\Z_{\ell}(a))/\tors\times {H}^{2d-i}(X,\Z_{\ell}(b))/\tors \to
\Z_{\ell},$$ induced by cup-product pairing is perfect, i.e., the
induced homomorphism
$$H^{i}(X,\Z_{\ell}(a))/\tors \to
\Hom_{\Z_{\ell}}({H}^{2d-i}(X,\Z_{\ell}(b))/\tors, \ \Z_{\ell})$$ is
an isomorphism of $\Z_{\ell}$-modules.
\end{cor}

In the next section we develop an elementary machinery that will
help us to deal with projective systems (and limits) of perfect
bilinear forms on finite $\Z/\ell^n\Z$-modules. Using it, we prove
Theorem \ref{main} in Section \ref{proofs}.

I proved Corollary \ref{proper} thirty years ago when I was
contemplating \cite{ZarhinIzv1982}. I stated Corollary \ref{proper}
and sketched its proof in my letters to Peter Schneider (April 22
and July 15, 1981, see \cite[Added in Proof on p. 142]{Schneider}).
Recently I learned from Alexei Skorobogatov that this result is
still missing in the literature but continues to be of  interest
\cite{Skorobogatov}. That is why I decided to publish it.

I am grateful to Alexey Parshin, Peter Schneider and Alexei
Skorobogatov for their interest in this subject. My special thanks go to Nick Katz and P. Deligne for stimulating comments. I am grateful to the referee, whose comments helped to improve the exposition.

The final version of this paper was prepared during my stay at the Weizmann Institute of Science in May--June of 2012: I am grateful to its Department of Mathematics for the hospitality.

\section{Linear algebra}
\label{algebra}

\begin{sect}
Let $E$ be a complete discrete valuation field, $\Lambda \subset E$
the corresponding discrete valuation ring with maximal ideal $\mm$.
Let $\pi \in \mm$ be an uniformizer, i.e., $\mm=\pi\Lambda$.
Clearly, $\Lambda$ coincides with the projective limit of its
quotients, the local rings
$$\Lambda_n=\Lambda/\mm^n=\Lambda/\pi^n\Lambda.$$
All $\Lambda_n$ carry the natural structure of $\Lambda$-modules. If
$M$ and $N$ are $\Lambda_n$-modules then
$$\Hom_{\Lambda}(M,N)=\Hom_{\Lambda_n}(M,N).$$
In addition, if $j$ is an integer such that $j\ge n$  then both $M$ and $N$ carry the
natural structure of $\Lambda_j$-modules and
$$\Hom_{\Lambda}(M,N)=\Hom_{\Lambda_n}(M,N)=\Hom_{\Lambda_j}(M,N).$$
There is a (non-canonical) isomorphism of $\Lambda_j$-modules
$$\Lambda_n \cong
\Hom_{\Lambda_j}(\Lambda_n,\Lambda_j)=\Hom_{\Lambda}(\Lambda_n,\Lambda_j)$$
that sends $1 \in \Lambda_n$ to the homomorphism
$$\Lambda_n \to \Lambda_j, \ \lambda+\pi^n\Lambda\mapsto \pi^{j-n}\lambda+\pi^j\Lambda.$$
If $U$ is a free $\Lambda$-module of finite rank then we write $\rk(U)$ for its rank.
\begin{lem}
\label{value2} Let $$\beta: U \times V \to \Lambda$$ be a
$\Lambda$-bilinear pairing between  free $\Lambda$-modules $U$ and
$V$ of finite rank.  
Let 
\begin{equation}
\label{Lmap}
q_{\beta}: V \to \Hom_{\Lambda}(U,\Lambda)\ v \mapsto \{u\mapsto \beta(u,v)\};
 \end{equation}
\begin{equation}
\label{Rmap}
s_{\beta}:U \to \Hom_{\Lambda}(V,\Lambda), \ u \mapsto \{v\mapsto \beta(u,v)\}
\end{equation}
be the  homomorphisms of $\Lambda$-modules attached to $\beta$.

Assume that  $\beta$ enjoys the following property.

If $u \in U\setminus \pi U$ then there exists $v\in V$ with
$\beta(u,v)=1$.

Then:

\begin{itemize}
\item[(i)]
 The homomorphism $q_{\beta}$  is surjective. 
\item[(ii)]
If $\rk(U)=\rk(V)$ then both 
homomorphisms $s_{\beta}$ and $q_{\beta}$ are isomorphisms, i.e., $\beta$ is a perfect pairing.
\end{itemize}
\end{lem}

\begin{proof}
Let us prove (i).
If $U=\{0\}$ then there is nothing to prove.
Now let us do induction by the rank of $U$. Assume that $U
\ne \{0\}$. Clearly, $U \ne \pi U$.
 Pick (indivisible) $u_0 \in U \setminus \pi U$. Then there exists $v_0\in V$
 with
$\beta(u_0,v_0)=1$. Put
$$U_1=\{u\in U\mid \beta(u,v_0)=0\}, \ V_1=\{v\in V\mid
\beta(u_0,v)=0\}.$$ Clearly,
$$U=U_1\oplus \Lambda u_0, \ V=V_1\oplus \Lambda v_0.$$
If $u\in U_1$ does not lie in $\pi U_1$ then it does not lie in $\pi
U=\pi U_1\oplus \pi\Lambda u_0$. Therefore there exists $v \in V$
with $\beta(u,v)=1$. We have $v=v_1+ \lambda v_0$ for some $v_1\in
V_1$ and $\lambda\in \Lambda$. Since $U_1$ and $v_0$ are orthogonal
with respect to $\beta$, $\beta(u,v_1)=1$.
 Now in order to prove the surjectiveness of $q_{\beta}$, one has only to apply the induction assumption to $\beta:
U_1\times V_1 \to \Lambda$, taking into account that $\rk(U_1)=\rk(U)-1$.

In order to prove (ii),  suppose that $\rk(U)=\rk(V)$ and denote this rank by $r$. Then $r$ is also the rank of the free $\Lambda$-module
$\Hom_{\Lambda}(U,\Lambda)$. So, 
$$q_{\beta}: U \to \Hom_{\Lambda}(V,\Lambda)$$ is a surjective homomorphism of free $\Lambda$-modules of the same  rank  $r$. The freeness of $\Hom_{\Lambda}(V,\Lambda)$ implies that  surjective  $q_{\beta}$ admits a ``section''
 and therefore the $\Lambda$-modules  $U$ and $\ker(q_{\beta})\oplus \Hom_{\Lambda}(V,\Lambda)$
 are isomorphic. Since $\ker(q_{\beta})$ is a $\Lambda$-submodule of free $U$ and $\Lambda$ is a principal ideal domain,  $\ker(q_{\beta})$  is also free and
 $$r=\rk(U)=\rk(\ker(q_{\beta}))+r,$$
 i.e.,  $\rk(\ker(q_{\beta}))=0$. This means that $\ker(q_{\beta})=\{0\}$, i.e., $q_{\beta}$ is injective. Since we already know that it is surjective, $q_{\beta}$ is an {\sl isomorphism}.
 
 It remains to check that $s_{\beta}$  is an isomorphism. If $r=0$, there is nothing to prove.
 So, assume that $r\ge 1$ and choose a basis $\{u_1, \dots, u_r\}$ of $U$.  Let 
 $\{u_1^{*}, \dots, u_r^{*}\}$ be the {\sl dual basis} in $\Hom_{\Lambda}(U,\Lambda)$.
 Since $q_{\beta}$ is an  isomorphism, there is a  basis $\{v_1, \dots, v_r\}$ of $V$ such that 
 $$q_{\beta}(v_i)=u_i^{*} \ \forall i \in \{1, \dots, r\},$$
   i.e.,
 $$\beta(u_i,v_i)=q_{\beta}(v_i)(u_i)=1; \ \
 \beta(u_i, v_j)=q_{\beta}(v_i)(u_j) =0 \ \forall j \ne i,   j\in \{1, \dots, r\}.$$
  Then
 $s_{\beta}:U \to \Hom_{\Lambda}(V,\Lambda)$ sends   every basis element $u_j$ to the homomorphism 
 $$s_{\beta}(u_i): \sum_{i=1}^r a_i v_i \mapsto a_j.$$
 Clearly, $\{s_{\beta}(u_i)\mid 1 \le i \le r\}$ is a basis in $\Hom_{\Lambda}(V,\Lambda)$, which implies that
 $s_{\beta}$  is an isomorphism. This ends the proof.
\end{proof}

\begin{sect}

Now and till the rest of this section we assume that $E$ is locally
compact, i.e., $E$ is either a finite algebraic extension of
$\Q_{\ell}$ or $\F_q((t))$. Then $\Lambda$ is compact and all  local
rings $\Lambda_n$ are finite. If a $\Lambda$-module $M$ is a finite
set then it is a finitely generated torsion $\Lambda$-module and
therefore is isomorphic to a finite direct sum of $\Lambda$-modules
$\Lambda_n$. (If $\pi^j M=\{0\}$ then   $n\le j$ for all such $n$). It
follows that if $\pi^j M=\{0\}$ then  the finite $\Lambda$-modules $M$ and
$\Hom_{\Lambda}(M,\Lambda_j)=\Hom_{\Lambda_j}(M,\Lambda_j)$ are
isomorphic; in particular, they have the same cardinality.
\end{sect}
\end{sect}

\begin{lem}
\label{n0} Let $M$  be a finite $\Lambda_n$-module. Let $x$ be an
element of $M$ such that $\pi^{n-1}x \ne 0$. Then there exists a
homomorphism of $\Lambda$-modules $\phi: M \to \Lambda_n$ with
$\phi(x)=1$.
\end{lem}

\begin{proof}
Clearly,  $y=\pi^{n-1}x$ is a nonzero element of $M$.
 Suppose the assertion of Lemma is not true. Then for all $\phi \in
\Hom_{\Lambda}(M,\Lambda_n)$ the image $\phi(x)$ is {\sl not} a unit
in $\Lambda_n$ and therefore lies in the maximal ideal $\pi
\Lambda_n$. It follows that
$$\phi(y) =\pi^{n-1}\phi(x)\in \pi^{n-1}\pi
\Lambda_n=\pi^n\Lambda_n=\{0\}.$$ This implies that every $\phi$
kills $y$ and therefore
$$\Hom_{\Lambda}(M,\Lambda_n)=\Hom_{\Lambda}(M/\Lambda
y,\Lambda_n).$$ Since $y\ne 0$,  the finite modules $M$ and
$M/\Lambda y$ have different orders and therefore their duals must
have different orders. Contradiction.
\end{proof}

The following 
definition
was inspired by the universal
coefficients theorem (UCT) for $\ell$-adic sheaves \cite[Ch. V, Sect. 1, Lemma 1.11 on p.
165]{Milne}.

\begin{defn}
\label{univcoeff}
 Let $\{T_n\}_{n=1}^{\infty}$ be a projective system of finite
$\Lambda$-modules $T_n$. Suppose that every $T_n$ is actually a
finite $\Lambda_n$-module. Suppose that the projective limit $T$ is
a finitely generated $\Lambda$-module. 

We say 
that $\{T_n\}$ satisfies  the  (UCT) condition if there exists a positive
integer $n_0$ such that for all $n\ge n_0$ the natural map
$$T/\pi^n T \to T_n$$ is injective.
\end{defn}

\begin{rem}
It was erroneously claimed in \cite[Lemma 2.5]{ZarhinVdG} that $\{T_n\}$   always satisfies the (UCT) condition.
\end{rem}

\begin{sect}
\label{s26}
Let us consider a projective system of triples $(H_n, T_n,e_n)$
of
the following type.

\begin{enumerate}
\item[(i)]
$H_n$ and $T_n$ are finite $\Lambda_n$-modules and the transition
maps
$$H_{n+1}\to H_n, \ T_{n+1} \to T_n$$
are homomorphisms of the corresponding $\Lambda$-modules.

\item[(ii)]
$e_n: H_n \times T_n \to \Lambda_n$ is a perfect pairing of finite
$\Lambda_n$-modules, i.e., the induced homomorphisms
$$H_n \to \Hom_{\Lambda_n}(T_n,\Lambda_n)=\Hom_{\Lambda}(T_n,\Lambda_n),$$
$$T_n \to
\Hom_{\Lambda_n}(H_n,\Lambda_n)=\Hom_{\Lambda}(H_n,\Lambda_n)$$ are
isomorphisms of finite $\Lambda_n$-modules. (In fact, the finiteness
implies that the first homomorphism is an isomorphism if and only if
the second one is an isomorphism.)

\item[(iii)] Both projective limits $H$ of $\{H_n\}$ and $T$ of
$\{T_n\}$ are {\sl finitely generated} $\Lambda$-modules. In
particular, $H/\tors$ and $T/\tors$ are free $\Lambda$-modules of
finite rank.
\end{enumerate}
\end{sect}

\begin{sect}
 Let us consider the  projective limit of $\{e_n\}$, which is the $\Lambda$-bilinear pairing
$$e: H \times T \to \Lambda.$$
In other words, if $h\in H$ corresponds to a sequence $\{h_n \in
H_n\}_{n=1}^{\infty}$ and $t\in T$ corresponds to a sequence $\{t_n
\in T_n\}_{n=1}^{\infty}$ then $e(h,t)\in \Lambda$ is the projective
limit of the sequence $e_n(h_n,t_n)\in \Lambda_n$. Clearly, $e$
kills $H_{\tors}$ and $T_{\tors}$ and therefore gives rise to the
$\Lambda$-bilinear pairing of free $\Lambda$-modules
$$\bar{e}: H/\tors \times T/\tors \to \Lambda.$$
\end{sect}

\begin{thm}
\label{perfect}
Let  $(H_n, S_n,e_n)$ be a projective system
satisfying the assumptions (i), (ii), (iii)  of Section \ref{s26}.
Suppose that  $\{H_n\}$ satisfies the   (UCT) condition.


If $\rk(H/\tors)=\rk(T/\tors)$ then the pairing
$$\bar{e}: H/\tors \times T/\tors \to \Lambda$$
is perfect, i.e., the corresponding homomorphisms of free
$\Lambda$-modules \eqref{Lmap}  and \eqref{Rmap} attached to $\beta:=\bar{e}$
$$H/\tors \to \Hom(T/\tors, \ \Lambda), \ T/\tors \to \Hom(H/\tors, \ \Lambda)$$
are isomorphisms.
\end{thm}

\begin{proof} We can choose a free $\Lambda$-submodule
$H^f$ in $H$ such that $H=H^f\oplus
H_{\tors}$. Similarly, we choose a free $\Lambda$-submodule $T^f$ in $T$ such that
$T=T^f\oplus T_{\tors}$.
Clearly,  there are isomorphisms of free $\Lambda$-modules
$$H/\tors \cong H^f, \ T/\tors \cong T^f.$$
In particular,  the ranks of $H^f$ and $T^f$ do {\sl coincide}.

In order to prove Theorem \ref{perfect}, it suffices to
check that the pairing
$$e: H^f \times T^f \to \Lambda$$
is perfect.

If $H^f=\{0\}$ and $T^f=\{0\}$ then there is nothing to prove. If both $H^f$ are
$T^f$ are nonzero modules then the theorem follows from Lemma \ref{value2} and the following lemma.
\begin{lem}
\label{value1}
Suppose that $H^f \ne\{0\}$. Then $H^f \ne \pi H^f$ and  for any $h \in
H^f\setminus \pi H^f$  there exists $t \in T^f$ with $e(h,t)=1$.


\end{lem}

\end{proof}

\begin{proof}[Proof of Lemma \ref{value1}]
Clearly, $H^f \ne \pi H^f$. Pick $h \in
H^f\setminus \pi H^f$. Then $\pi^{n-1}h$ does not belong to $\pi^n
H$ for all $n$. In other words, the image $h_n$ of $h$ in $H/\pi^n
H$ satisfies $$\pi^{n-1}h_n \ne 0.$$ 
Since $\{H_n\}$ satisfies the condition (UCT),
 there
exists a positive integer $n_0$ such that for all $n\ge n_0$ the
homomorphism of $\Lambda_n$-modules $H/\pi^n H\to H_n$ is injective.
So, we may view $H/\pi^n H$ as a submodule of $H_n$ and $h_n$ as an
element of $H_n$. By Lemma \ref{n0}, there exists $t_n \in T_n$ with
$$e_n(h_n,t_n)=1\in\Lambda_n.$$
So, for each $n\ge n_0$ the subset
$$S_n=\{t_n\in T_n\mid e(h_n,t_n)=1\}\subset T_n$$
is non-empty. It is also finite and therefore the projective limit
$S$ of the sequence $\{S_n\}$ is a non-empty subset of $T$. Pick $s\in S$. Clearly,
$e(h,s)\in \Lambda$ is the projective limit of the sequence $n
\mapsto 1\in \Lambda_n \ (n\ge n_0)$. Therefore $e(h,s)=1\in\Lambda$
and we are almost done. Since $T=T^f\oplus T_{\tors}$, we have
$$s=t+w, \ t \in T^f, w \in T_{\tors}.$$
Since $e(h,w)=0$, we have
$$e(h,t)=e(h,s)-e(h,w)=1-0=1.$$
\end{proof}

\begin{rem}
The conclusion of Theorem \ref{perfect} remains valid if one replaces the equality of ranks assumption by the assumption that $\{T_n\}$ also satisfies the condition (UCT). The proof essentially remains the same: one has only to apply Lemma \ref{value1} one more time, permuting $\{H_n\}$ and $\{T_n\}$, and use \cite[Lemma 2.2]{ZarhinVdG} (instead of Lemma 
\ref{value2}).

\end{rem}

\section{Proof of main results}
\label{proofs}
 Let us put $$E=\Q_{\ell}, \Lambda=\Z_{\ell},
\pi=\ell.$$ We get
$$\Lambda_n=\Z_{\ell}/\ell^n \Z_{\ell}=\Z/\ell^n\Z.$$
We want to deduce Theorem \ref{main}  from Theorem \ref{perfect}.
In the notation of Section \ref{algebra} we put
$$H_n=H^i(X,\Z/\ell^n\Z\ (a)), \ T_n=
H_c^{2d-i}(X,\Z/\ell^n\Z\ (b))$$ and
$$e_n:H^i(X,\Z/\ell^n\Z \ (a))\times
H_{c}^{2d-i}(X,\Z/\ell^n\Z\ (b))\to
H_{c}^{2d}(X,\Z/\ell^n\Z\ (d))=\Z/\ell^n\Z$$ the pairing induced by
cup-product. By definition, the projective limit $H$ of the sequence
$H_n=H^i(X,\Z/\ell^n\Z(a))$ is $H^i(X,\Z_{\ell}(a))$ and the
projective limit $T$ of the sequence $T_n=H_c^{2d-i}(X,\Z/\ell^n\Z(b))$ is
$H_c^{2d-i}(X,\Z_{\ell}(b))$. In the notation of Theorem \ref{perfect}, we need to check the {\sl unimodularity} of $\bar{e}$.

Recall that all the $\ell$-adic cohomology groups $H^i(X,\Z/\ell^n\Z (a))$ and
$H_{c}^{2d-i}(Z,\Z/\ell^n\Z(b))$ are finite $\Z/\ell^n\Z$-modules
and all  the $\Z_{\ell}$-modules $H^i(X,\Z_\ell(a))$ and
$H_c^{2d-i}(X,\Z_{\ell}(b))$ are finitely generated. These
finiteness results are fundamental finiteness theorems in \'etale
cohomology from {\bf SGA 4, 4$\frac{1}{2}$, 5}, see \cite[pp.
22--24]{Katz} for precise references. The perfectness of the pairing
$e_n$ is the Poincar\'e duality in \'etale cohomology (see \cite[Ch.
VI, Sect. 11, Cor. 11.2 on p. 276]{Milne}, \cite[p. 23]{Katz}).
Applying the universal coefficients theorem \cite[Ch. V, Sect. 1, Lemma 1.11 on p.
165]{Milne} to the $\ell$-adic sheaf $\Z_{\ell}(a)$,  we obtain that the natural homomorphism 
$$H/\ell^n H=H^i(X,\Z_{\ell}(a))/\ell^n \to H^i(X,\Z/\ell^n\Z\ (a))=H_n$$ is {\sl injective}
for all positive integers $n$.
This implies that 
the projective system $H_n=H^i(X,\Z/\ell^n\Z\ (a))$ satisfies the  (UCT) condition.
Recall (Subsection \ref{prelimX}) that the ranks of free $\Z_{\ell}$-modules 
$$H/\tors=H^i(X,\Z_{\ell} (a))/\tors \ \text{ and } \ T/\tors=H_c^{2d-i}(X,\Z_\ell(b))/\tors$$ do {\sl coincide}.

 Now Theorem \ref{main} follows from Theorem \ref{perfect}.

\begin{rem}
P. Deligne pointed out that one may deduce Theorem \ref{main} from the  derived category version of Poincar\'e duality 
\cite[Sect. 3.2.6]{Deligne}, \cite[Th. 6.3]{Ekedahl}.
\end{rem}

\end{document}